\documentclass[11pt,a4paper,reqno]{amsart}

\usepackage{amssymb,mathtools}
\usepackage{enumerate}
\usepackage{hyperref}

\newtheorem{theorem}{Theorem}[section]

\newtheorem{lemma}[theorem]{Lemma}
\newtheorem{corollary}[theorem]{Corollary}
\theoremstyle{definition}

\newtheorem{remark}[theorem]{Remark}

\newtheorem{notation}[theorem]{Notation}
\newtheorem{question}[theorem]{Question}

\newcommand{\Sym}{\mathrm{Sym}}
\newcommand{\bbF}{\mathbb{F}}
\newcommand{\Cen}{\mathbf{C}}
\newcommand{\Nor}{\mathbf{N}}
\newcommand\GL{\mathrm{GL}}

\allowdisplaybreaks

\title{Asymptotic enumeration of minimally transitive permutation groups}

\author[B.~Xia]{Binzhou Xia}
\address[Binzhou Xia]{School of Mathematics and Statistics\\The University of Melbourne\\Parkville, VIC 3010\\Australia}
\email{binzhoux@unimelb.edu.au}

\author[S.~Zheng]{Shasha Zheng}
\address[Shasha Zheng]{Alfr\'ed R\'enyi Institute of Mathematics\\Budapest\\Re\'altanoda u.~13-15, H-1053\\Hungary}
\email{shashazheng@renyi.hu}

\date{}

\begin{document}

\begin{abstract}
We prove that Pyber's upper bound $2^{O(n\log(n))}$ for the number of minimally transitive subgroups of $S_n$ is best possible along the powers of every fixed prime, even when the groups are counted up to permutational isomorphism. As a byproduct, our construction shows that, along the powers of every fixed prime, the maximum order of a minimally transitive permutation group of degree $n$ is $2^{\Theta(n)}$. For completeness, we also present Pyber's previously unpublished proof of his upper bound. We further deduce that the numbers of labelled vertex-transitive graphs and digraphs of order $n$ are both $2^{\Theta(n\log(n))}$, and discuss the implications of our results for approaches to the McKay--Praeger conjecture.


\end{abstract}

\maketitle

\section{Introduction}

Throughout this paper, $\log$ denotes the logarithm to base $2$, and we use the standard asymptotic notation $O$, $\Omega$, and $\Theta$ (see Section~\ref{Sec:Pre}).

A transitive permutation group is said to be \emph{minimally transitive} if none of its proper subgroups is transitive. Every finite transitive permutation group contains a minimally transitive subgroup, and hence minimally transitive groups arise naturally as reduction objects in various applications~\cite{BS1985,NV1977,SW1963,Tracey2018}.

Several structural and generation questions concerning minimally transitive permutation groups have been studied extensively. For structural results, see~\cite{DVS2007} and the references therein, as well as~\cite[Section~3]{Tracey2018}. Concerning generation, Neumann and Vaughan-Lee~\cite[Lemma~3.1]{NV1977} proved that every minimally transitive permutation group of degree $n$ can be generated by at most $\log(n)$ elements. Refined versions of this result depending on the prime factorization $n=p_1^{m_1}\cdots p_s^{m_s}$ have been studied by Shepperd and Wiegold~\cite{SW1963}, Lucchini~\cite{Lucchini1998}, and Tracey~\cite{Tracey2016}. In particular, Tracey proved that every minimally transitive permutation group of degree $n$ can be generated by at most $\max\{m_1,\ldots,m_s\}+1$ elements.

In this paper, we are concerned with the enumeration of minimally transitive permutation groups, a problem motivated particularly by applications to the enumeration of vertex-transitive combinatorial structures.
For example, in attempting to improve their upper bound on the number of labelled vertex-transitive graphs, Babai and S\'{o}s~\cite{BS1985} observed that it was a major unsolved problem in group theory to estimate the number of minimally transitive permutation groups of a given degree $n$. The first upper bound in this direction, given in~\cite[Corollary~6.5]{BS1985}, is of the form $2^{(1+o(1))n(\log(n))^2}$. Pyber subsequently improved this to the following upper bound.

\begin{theorem}[Pyber]\label{Thm:UpperBound}
As the integer $n\to\infty$, the number of minimally transitive subgroups of $S_n$ is $2^{O(n\log(n))}$.
\end{theorem}

This result was stated in~\cite[Theorem~4.4]{Pyber1991}, where it was claimed more precisely that the number of minimally transitive subgroups of $S_n$ is at most $(n!)^{4+o(1)}$. Note that $(n!)^{4+o(1)}=2^{\Theta(n\log(n))}$ by Stirling's formula. Pyber communicated to us the proof presented in Section~\ref{Sec:UpperBound}, and a more precise version that follows from this proof is stated as Theorem~\ref{Thm:UpperBoundDetail}.

How close is Pyber's upper bound to being sharp?
The problem of looking for a sharp upper bound for the number of conjugacy classes of minimally transitive groups was proposed in~\cite{Pyber1991}, and it was commented that an exponential upper bound would be quite sharp and also give the order of magnitude of the number of vertex-transitive graphs. Similar observations were made by Morris and Spiga~\cite[Subsection~8.4]{MS2021} in connection with the asymptotic enumeration of vertex-transitive digraphs.

Our main result shows that Pyber's upper bound $2^{O(n\log(n))}$ is best possible along the powers of every fixed prime, even when the minimally transitive groups are counted only up to permutational isomorphism. In particular, an upper bound of the form $2^{O(n)}$ cannot hold. Here two subgroups of $S_n$ are said to be \emph{permutationally isomorphic} if they are conjugate in $S_n$, or equivalently, if their natural permutation representations differ only by a relabelling of the underlying set.

\begin{theorem}\label{Thm:LowerBound}
For each fixed prime number $p$, as $n$ increases as a power of $p$, the number of permutational isomorphism classes of minimally transitive permutation groups of degree $n$ is $2^{\Omega(n\log(n))}$.
\end{theorem}

A more precise version of the above theorem that we will prove is stated as Theorem~\ref{Thm:LowerBoundDetail}.

\begin{remark}
Theorems~\ref{Thm:UpperBound} and~\ref{Thm:LowerBound} together determine the correct order of magnitude of the exponent along the powers of any fixed prime: this exponent is $\Theta(n\log(n))$ no matter we count the groups up to permutational isomorphism or not. Using product actions, the construction underlying Theorem~\ref{Thm:LowerBound} also produces many minimally transitive groups in further degrees. However, such a matching lower bound up to permutational isomorphism cannot hold in every degree. Indeed, if $n=p$ is prime, then every transitive subgroup of $S_p$ contains a regular cyclic subgroup of order $p$. In this case, the minimal transitivity forces the group itself to be cyclic and regular, and hence there is only one permutational isomorphism class of minimally transitive groups of degree $p$.
\end{remark}

\begin{remark}
Every regular permutation group is minimally transitive, and the permutational isomorphism classes of regular groups of degree $n$ are in bijection with the isomorphism classes of groups of order $n$. Nevertheless, regular groups alone are far too sparse to account for Theorem~\ref{Thm:LowerBound}. For $n=p^m$, the known asymptotic enumeration of finite $p$-groups gives only $2^{\Theta((\log(n))^3)}$ permutational isomorphism classes of regular permutation groups, by results of Higman and Sims~\cite{Higman1960,Sims1965}. This is negligible compared with $2^{\Omega(n\log(n))}$. Thus, the lower bound in Theorem~\ref{Thm:LowerBound} is supplied overwhelmingly by nonregular minimally transitive groups.
\end{remark}

The construction used to prove Theorem~\ref{Thm:LowerBound} also gives, as a byproduct, minimally transitive permutation groups of large order.

\begin{theorem}\label{Thm:LargeOrder}
Let $p$ be a prime. For every $p$-power $n\geq p^2$, there exists a minimally transitive permutation group of degree $n$ and order $np^{n/p-p}$.
\end{theorem}

\begin{remark}\label{Rem:LargeOrder}
Let $p$ be a prime, and let $n$ be a $p$-power. For a transitive permutation group of degree $n$, its Sylow $p$-subgroups are also transitive. Consequently, the minimally transitive permutation groups of degree $n$ are all $p$-groups, and hence have order at most $|S_n|_p=p^{(n-1)/(p-1)}$.
Together with Theorem~\ref{Thm:LargeOrder}, this shows that, along the powers of each fixed prime $p$, the maximum order of a minimally transitive permutation group of degree $n$ is $2^{\Theta(n)}$.
\end{remark}

Theorem~\ref{Thm:LargeOrder} is proved in Section~\ref{Sec:Remark}. In the same section, we also discuss what the preceding results imply about the asymptotic enumeration of vertex-transitive graphs and digraphs. One immediate consequence of Theorem~\ref{Thm:UpperBound} is the following corollary, which improves the upper bound $2^{(1+o(1))n(\log(n))^2}$ of Babai and S\'{o}s~\cite[Theorem~6.1]{BS1985} to the correct order of magnitude in the exponent. Here, by a digraph (including graph as a special case) we mean one whose arc set is an arbitrary set of ordered pairs of vertices, while a labelled digraph of order $n$ has the fixed vertex set $\{1,\ldots,n\}$, and two such labelled digraphs are regarded as distinct whenever their arc sets differ.

\begin{corollary}\label{Cor:VT}
As the integer $n\to\infty$, the number of labelled vertex-transitive graphs of order $n$ is $2^{\Theta(n\log(n))}$, and so is the number of labelled vertex-transitive digraphs of order $n$.
\end{corollary}

The Classification of Finite Simple Groups enters our argument only through Theorem~\ref{Lem:AG1982}, which is used only in the proof of Theorem~\ref{Thm:UpperBound}. All other arguments in the paper, including the entire proof of the lower bound, are independent of the Classification.

\section{Preliminaries}\label{Sec:Pre}

In a group, we use the convention $x^y=y^{-1}xy$ for conjugation and $[x,y]=x^{-1}y^{-1}xy$ for commutator.
For sets $A$ and $B$, denote the set of functions from $A$ to $B$ by $\mathrm{Fun}(A,B)$.
For nonnegative functions $f$ and $g$ with $g(n)>0$ for all sufficiently large $n$, we write $f(n)=O(g(n))$ if there exists a constant $C>0$ such that $f(n)\leq Cg(n)$ for all sufficiently large $n$, and $f(n)=\Omega(g(n))$ if there exists a constant $c>0$ such that $f(n)\geq cg(n)$ for all sufficiently large $n$. We write $f(n)=\Theta(g(n))$ if $f(n)=O(g(n))$ and $f(n)=\Omega(g(n))$.

Recall the following elementary probabilistic inequality.

\begin{lemma}[Chebyshev's inequality]\label{Lem:Chebyshev}
Let $X$ be a real-valued random variable with finite mean $\mu$ and finite standard deviation $\sigma$. Then for every real number $a>0$,
\[
\Pr(|X-\mu|\geq a)\leq\frac{\sigma^2}{a^2}.
\]
\end{lemma}

The Frattini subgroup of a group $G$ is denoted by $\Phi(G)$.
The following group-theoretic observation is well known.

\begin{lemma}\label{Lem:Frattini}
Let $G$ be a transitive permutation $p$-group with prime $p$, and let $H$ be a point stabilizer in $G$. Then $G$ is minimally transitive if and only if $H\leq\Phi(G)$.
\end{lemma}

\begin{proof}
For a maximal subgroup $M$ of $G$, since $M$ is normal, $MH\neq G$ if and only if $H\leq M$. Thus, $G$ has no transitive maximal subgroup if and only if $H$ is contained in every maximal subgroup of $G$. Hence the lemma follows.
\end{proof}

In the next lemma, a maximal solvable subgroup means a maximal one (with respect to set containment) among solvable subgroups.

\begin{lemma}[{\cite[Lemma~3.2(iii)]{Pyber1993}}]\label{Lem:Pyber1993}
For a positive integer $n$, the number of conjugacy classes of maximal solvable subgroups of $S_n$ is at most $2^{17n}$.
\end{lemma}

We also need the Dixon's bound on the orders of solvable permutation groups.

\begin{theorem}[{\cite[Theorem~3]{Dixon1967}}]\label{Lem:Dixon1967}
For a positive integer $n$, every solvable subgroup of $S_n$ has order at most $24^{(n-1)/3}$.
\end{theorem}

The following is~\cite[Theorem~A]{AG1982}, whose proof used the Classification of Finite Simple groups.

\begin{theorem}[Aschbacher--Guralnick]\label{Lem:AG1982}
Every finite group $G$ has a solvable subgroup $H$ and an element $g\in G$ such that $G=\langle H,H^g\rangle$.
\end{theorem}

\section{Proof of Theorem~\ref{Thm:UpperBound}}\label{Sec:UpperBound}

\begin{lemma}\label{Lem:AGplus}
Every finite minimally transitive group $G$ has an orbit-minimal solvable subgroup $H$ and an element $g\in G$ such that $G=\langle H, g\rangle$.
\end{lemma}

\begin{proof}
By Theorem~\ref{Lem:AG1982}, the finite $G$ is generated by some solvable subgroup $H$ and element $g\in G$.
Choose $H$ of minimum order subject to the existence of some $g\in G$ with $G=\langle H,g\rangle$, and fix such an element $g$.
To prove the lemma, we only need to show that $H$ is orbit-minimal.
Suppose for a contradiction that there exists a proper subgroup $K$ of $H$ that has the same orbits as $H$.
Then since $\langle H,g\rangle=G$ is transitive, it follows that $\langle K,g\rangle$ is also transitive. By the minimal transitivity of $G$, we must have $\langle K,g\rangle=G$. However, since $K$ is also solvable, this contradicts the minimality of $|H|$. This completes the proof.
\end{proof}

A permutation group $H$ is called \emph{orbit-minimal} if every proper subgroup of $H$ has strictly more orbits than $H$ has. The following theorem is not only needed in our proof of Theorem~\ref{Thm:UpperBound} but also of independent interest. Its full statement of Theorem~\ref{Lem:Lovasz} is due to L\'{a}szl\'{o} Lov\'{a}sz and was recorded in~\cite[Theorem~1.5]{Pyber1991}. Since no proof was given there, we include an elementary proof obtained by adapting the argument of Neumann and Vaughan-Lee~\cite{NV1977}.

\begin{theorem}[Lov\'{a}sz]\label{Lem:Lovasz}
For a positive integer $n$, every orbit-minimal subgroup of $S_n$ with $k$ orbits can be generated by at most $\log(n-k+1)$ elements.
\end{theorem}

\begin{proof}
We imitate the proof of~\cite[Lemma~3.1]{NV1977}.
Let $X$ be an orbit-minimal subgroup of $S_n$ whose orbits are $\Omega_1,\ldots,\Omega_k$. Put $x_0=1$, and suppose that $x_0,x_1,\ldots,x_s$ have been selected for some nonnegative integer $s$. Let $X_s=\langle x_0,x_1,\ldots,x_s\rangle$, and let $\Delta_1,\ldots,\Delta_{t_s}$ be the orbits of $X_s$.

For each $i\in\{1,\ldots,t_s\}$, let $j_i$ be the unique element of $\{1,\ldots,k\}$ such that $\Delta_i\subseteq\Omega_{j_i}$, and let $Y_i$ be the stabilizer of $\Delta_i$ in $X$. The index $m_i$ of $Y_i$ in $X$ is the number of distinct $X$-translates of $\Delta_i$. These translates cover $\Omega_{j_i}$, as $X$ is transitive on $\Omega_{j_i}$, and so $m_i|\Delta_i|\geq|\Omega_{j_i}|$. Consequently,
\[
\sum_{i=1}^{t_s}\frac{1}{m_i}
\leq\sum_{i=1}^{t_s}\frac{|\Delta_i|}{|\Omega_{j_i}|}
=\sum_{j=1}^k\frac{1}{|\Omega_j|}\sum_{\substack{i=1\\j_i=j}}^{t_s}|\Delta_i|
\leq\sum_{j=1}^k\frac{1}{|\Omega_j|}\cdot|\Omega_j|
=k,
\]
and it follows that
\begin{equation}\label{Eqn:7}
\sum_{i=1}^{t_s}|Y_i|=\sum_{i=1}^{t_s}\frac{|X|}{m_i}\leq k|X|.
\end{equation}

For $x\in X$, let $\ell(x)=|\{i\in\{1,\ldots,t_s\}\mid x\in Y_i\}|$ be the number of $X_s$-orbits that are stabilized by $x$. Double counting $(x,i)$ with $x\in Y_i$, and then applying~\eqref{Eqn:7}, we obtain
\[
\sum_{x\in X}\ell(x)=\sum_{i=1}^{t_s}|Y_i|\leq k|X|.
\]
Thus the average value of $\ell(x)$ for all $x\in X$ is at most $k$.
Suppose that $t_s>k$. Then since $\ell(1)=t_s>k$, there exists $x_{s+1}\in X$ such that $\ell(x_{s+1})\leq k-1$. Let
\[
X_{s+1}=\langle X_s,x_{s+1}\rangle.
\]
Each $X_{s+1}$-orbit is a union of $X_s$-orbits. Moreover, an $X_{s+1}$-orbit consists of a single $X_s$-orbit $\Delta_i$ if and only if $x_{s+1}\in Y_i$. Therefore, at most $k-1$ orbits of $X_{s+1}$ consist of a single $X_s$-orbit. Since every other $X_{s+1}$-orbit contains at least two $X_s$-orbits, it follows that
\[
t_s\geq t_{s+1}+\big(t_{s+1}-(k-1)\big),
\]
and hence
\begin{equation}\label{Eqn:8}
t_{s+1}-k+1\leq\frac{t_s-k+1}{2}.
\end{equation}

Initially, $X_0=1$ has $t_0=n$ orbits. Repeated application of~\eqref{Eqn:8} shows that
\[
t_s-k+1\leq\frac{n-k+1}{2^s}
\]
for each $s$ such that $t_{s-1}>k$. Let $d=\lfloor\log(n-k+1)\rfloor$. If $t_d>k$, then
\[
1\leq t_{d+1}-k+1\leq\frac{n-k+1}{2^{d+1}},
\]
a contradiction. Thus the number $t_d$ of orbits of $X_d$ is $k$. Since $X$ is orbit-minimal, this forces
\[
X=X_d=\langle x_0,x_1,\ldots,x_d\rangle,
\]
and so $X$ can be generated by $d$ elements, as $x_0=1$.
\end{proof}

We now give the upper-bound argument communicated to us by Pyber, keeping track of the constants. It yields the following result, which immediately implies Theorem~\ref{Thm:UpperBound}.

\begin{theorem}\label{Thm:UpperBoundDetail}
For positive integers $n$, the number of minimally transitive subgroups of $S_n$ is at most $2^{(3+\log(3)/3+o(1))n\log(n)}$.
\end{theorem}

\begin{proof}
We follow the approach in the proof of~\cite[Theorem~4.2]{Pyber1991}. By Lemma~\ref{Lem:Pyber1993}, the number of conjugacy classes of maximal solvable subgroups of $S_n$ is at most $2^{17n}$, implying that there are at most
\[
n!2^{17n}
\]
such subgroups of $S_n$ in total. By Theorem~\ref{Lem:Dixon1967}, any such maximal solvable subgroup $M$ of $S_n$ has order at most $24^{(n-1)/3}$. Moreover, Theorem~\ref{Lem:Lovasz} implies that any orbit-minimal subgroup $H$ of $S_n$ can be generated by $\log(n)$ elements. Thus, for a fixed $M$, the number of possibilities for $H$ is at most $|M|^{\log(n)}\leq 24^{(n-1)\log(n)/3}$.
Consequently, there are at most
\[
n!2^{17n}\cdot 24^{(n-1)\log(n)/3}
\]
orbit-minimal solvable subgroups of $S_n$. Hence we conclude by Lemma~\ref{Lem:AGplus} and Stirling's formula that the number of minimally transitive subgroups of $S_n$ is at most
\[
n!2^{17n}\cdot 24^{(n-1)\log(n)/3}\cdot n!=2^{(3+\log(3)/3+o(1))n\log(n)}.\qedhere
\]
\end{proof}

\section{Proof of Theorem~\ref{Thm:LowerBound}}\label{Sec:LowerBound}

\begin{notation}\label{Ntn:1}
Let $p$ be a prime, let $m\geq2$ be an integer, let $V$ be a vector space over $\bbF_p$ with a basis $e_1,\ldots,e_m$, and for each $i\in\{1,\ldots,m\}$, let
\[
t_i\colon\ V\to V,\ \ v\mapsto v+e_i
\]
be the translation by $e_i$ on $V$. With
\[
T\coloneqq\langle t_1,\ldots,t_m\rangle,
\]
consider the imprimitive wreath product $\bbF_p^+\wr T=B\rtimes T$ in its action on $V\times\bbF_p$, where the action of $\bbF_p^+$ on $\bbF_p$ is by addition, and
\[
B\coloneqq\mathrm{Fun}(V,\bbF_p^+)\cong(\bbF_p^+)^{p^m}
\]
is the base group with each $b\in B$ acting on $V\times\bbF_p$ by
\[
(v,a)\mapsto(v,a+b(v)).
\]
Then $B\rtimes T$ is an imprimitive transitive permutation group on $V\times\bbF_p$ with a block system
\[
\mathcal{P}\coloneqq\{\{v\}\times\bbF_p\mid v\in V\}.
\]
Let $b=(b_1,\ldots,b_m)\in B^m$, and for each $i\in\{1,\ldots,m\}$, let $g_i=b_it_i\in B\rtimes T$, namely
\[
g_i\colon\ V\times\bbF_p\to V\times\bbF_p,\ \ (v,a)\mapsto(v+e_i,a+b_i(v)).
\]
Let $G=G(b)=\langle g_1,\ldots,g_m\rangle$ and $K=K(b)=G(b)\cap B$.
\end{notation}

Under Notation~\ref{Ntn:1}, the natural projection $B\rtimes T\to T$ restricts to an epimorphism $G\to T$ with kernel $K$, and so the induced permutation group $G/K$ of $G$ on $\mathcal{P}$ is $T$.

\begin{lemma}\label{Lem:K}
Under Notation~$\ref{Ntn:1}$, $K=\Phi(G)$ is the normal closure of
\begin{equation}\label{Eqn:4}
\{g_i^p\mid1\leq i\leq m\}\cup\{[g_i,g_j]\mid1\leq i<j\leq m\}
\end{equation}
in $G$. Moreover, the following are equivalent:
\begin{enumerate}[\rm(a)]
\item\label{Lem:K:a} $K$ is nontrivial;
\item\label{Lem:K:b} $\mathcal{P}$ is the set of orbits of $K$ on $V\times\bbF_p$;
\item\label{Lem:K:c} $G$ is transitive on $V\times\bbF_p$;
\item\label{Lem:K:d} $G$ is minimally transitive on $V\times\bbF_p$.
\end{enumerate}
\end{lemma}

\begin{proof}
Let $\rho\colon G\to T$ be the restriction of the natural projection $B\rtimes T\to T$ to $G$. Then $\rho(g_i)=t_i$ for all $i\in\{1,\ldots,m\}$. Let $F$ be the free group on \(x_1,\ldots,x_m\).
By the universal property of $F$, the assignments $x_i\mapsto g_i$ and $x_i\mapsto t_i$ uniquely extend to homomorphisms $\varphi\colon F\to G$ and $\psi\colon F\to T$, respectively. Since $\rho\circ\varphi$ and $\psi$ agree on the free basis $x_1,\ldots,x_m$, they are equal. Since $T$ is an elementary abelian $p$-group, it has presentation
\[
\langle t_1,\ldots,t_m\mid t_i^p=1\text{ for }1\leq i\leq m,\ [t_i,t_j]=1\text{ for }1\leq i<j\leq m\rangle,
\]
and so $\mathrm{Ker}(\psi)$ is the normal closure of
\[
\{x_i^p\mid1\leq i\leq m\}\cup\{[x_i,x_j]\mid1\leq i<j\leq m\}
\]
in $F$. Note that $\varphi$ is surjective as \(G=\langle g_1,\ldots,g_m\rangle\). It follows that $K=\mathrm{Ker}(\rho)=\varphi(\mathrm{Ker}(\psi))$ is the normal closure of~\eqref{Eqn:4} in $G$. Since $G$ is a $p$-group, $\Phi(G)=G^pG'$. Thus $g_i^p\in\Phi(G)$ for all $1\leq i\leq m$, and $[g_i,g_j]\in\Phi(G)$ for all $1\leq i<j\leq m$. Consequently, $K\leq\Phi(G)$. Moreover, since $G/K\cong T$ is elementary abelian, $\Phi(G)\leq K$. This leads to $K=\Phi(G)$.

Next, we show~\eqref{Lem:K:a}$\Leftrightarrow$\eqref{Lem:K:b}$\Leftrightarrow$\eqref{Lem:K:c}. Note that the overgroup $B$ of $K$ has $\mathcal{P}$ as the set of orbits on $V\times\bbF_p$. If $K$ is nontrivial as~\eqref{Lem:K:a} asserts, then $K$ has an orbit $P$ on $V\times\bbF_p$ for some $P\in\mathcal{P}$. In this case, since $K$ is normal in $G$ while $G$ is transitive on $\mathcal{P}$, it follows that $\mathcal{P}$ is the set of orbits of $K$ on $V\times\bbF_p$. Thus,~\eqref{Lem:K:a}$\Rightarrow$\eqref{Lem:K:b}. The implication~\eqref{Lem:K:b}$\Rightarrow$\eqref{Lem:K:c} follows from the transitivity of $G$ on $\mathcal{P}$. If $G$ is transitive on $V\times\bbF_p$ as~\eqref{Lem:K:c} states, then $|K||T|=|G|\geq|V\times\bbF_p|=p|T|$, and so $K$ is nontrivial. Therefore,~\eqref{Lem:K:a}$\Leftrightarrow$\eqref{Lem:K:b}$\Leftrightarrow$\eqref{Lem:K:c} is valid.

Finally, let $H$ be the stabilizer in \(G\) of \((0,0)\in V\times\bbF_p\). Then $H\leq K$ as $T$ acts regularly on $\mathcal{P}$. Hence we derive from $K\leq\Phi(G)$ that $H\leq\Phi(G)$. This together with Lemma~\ref{Lem:Frattini} shows~\eqref{Lem:K:c}$\Leftrightarrow$~\eqref{Lem:K:d}.
\end{proof}

For the remainder of the proof, we adopt a standard viewpoint from modular representation theory, interpreting the action of $T$ on the base group $B$ as multiplication in a graded algebra. The following notation makes this precise.

\begin{notation}\label{Ntn:2}
Adopt Notation~$\ref{Ntn:1}$.
With $T$ acting on $B$ by conjugation, the elementary abelian $p$-group $B$ can be viewed as an $\bbF_p[T]$-module.
Moreover, $\bbF_p[T]$ itself is an $\bbF_p[T]$-module with the multiplication action of $T$.
For $u\in V$, let $\tau_u\in T$ be the translation $v\mapsto v+u$ on $V$. Define $\theta\colon \bbF_p[T]\to B$ by letting
\[
\theta\left(\sum_{u\in V} a_u\tau_u\right)(v)=a_v.
\]
Then $\theta$ is an isomorphism of $\bbF_p[T]$-modules. Let
\[
R=\bbF_p[X_1,\ldots,X_m]/(X_1^p,\ldots,X_m^p).
\]
Note that $1+X_i$ is invertible in $R$ for each $i\in\{1,\ldots,m\}$, as $(1+X_i)^p=1+X_i^p=1$.
We make $R$ an $\bbF_p[T]$-module by letting $t_i$ act as multiplication by $(1+X_i)^{-1}$ for $i\in\{1,\ldots,m\}$. The assignment
\[
X_i\mapsto t_i^{-1}-1
\]
extends to an algebra isomorphism $\eta\colon R\to \bbF_p[T]$. With the above action of $T$ on $R$, this is also an isomorphism of $\bbF_p[T]$-modules. Through the composite isomorphism
\[
R\xrightarrow{\eta}\bbF_p[T]\xrightarrow{\theta}B,
\]
we may identify $B$ with $R$ as an $\bbF_p[T]$-module. For $d\in\{0,1,\ldots,(p-1)m\}$, let $R_d$ be the $\bbF_p$-subspace of $R$ spanned by the monomials in $X_1,\ldots,X_m$ of degree $d$. Finally, define an $\bbF_p$-linear transformation
\[
\Psi\colon\ R^m\to R^{m(m+1)/2},\ \
b\mapsto\Big(\big(X_i^{p-1}b_i\big)_{1\leq i\leq m},\big(X_jb_i-X_ib_j\big)_{1\leq i<j\leq m}\Big).
\]
\end{notation}

\begin{lemma}\label{Lem:RelationMap}
Under Notation~$\ref{Ntn:2}$, $K(b)$ is the ideal of $R$ generated by the coordinates of $\Psi(b)$ for each $b\in R^m$.
\end{lemma}

\begin{proof}
By Lemma~\ref{Lem:K}, $K=K(b)$ is the normal closure in $G$ of the set~\eqref{Eqn:4}. Since $B$ is abelian, conjugation by $g_i=b_it_i$ on $B$ is the same as conjugation by $t_i$ for each $i\in\{1,\ldots,m\}$. Accordingly, $K$ is the $\bbF_p[T]$-submodule of $B$ generated by~\eqref{Eqn:4}. Recall from Notation~\ref{Ntn:2} that, under the identification of $B$ with $R$, the action of $t_i$ is multiplication by $(1+X_i)^{-1}$. Therefore, $K$ is the ideal of $R$ generated by~\eqref{Eqn:4}. Moreover, for $1\leq i\leq m$,
\begin{align*}
g_i^p&=(b_it_i)^p\\
&=b_i(t_ib_it_i^{-1})(t_i^2b_it_i^{-2})\cdots(t_i^{p-1}b_it_i^{1-p})\\
&=b_i+b_i^{t_i^{p-1}}+b_i^{t_i^{p-2}}+\cdots+b_i^{t_i}\\
&=\Bigg(\sum_{j=0}^{p-1}(1+X_i)^{-j}\Bigg)b_i\\
&=\Bigg(\sum_{j=0}^{p-1}(1+X_i)^j\Bigg)b_i\qquad\mbox{as $(1+X_i)^p=1$ in $R$}\\
&=X_i^{p-1}b_i\qquad\mbox{as $\sum_{j=0}^{p-1}(1+X_i)^j=X_i^{p-1}$ in $\bbF_p[X_i]$},
\end{align*}
and for $1\leq i<j\leq m$,
\begin{align*}
[g_i,g_j]&=g_i^{-1}g_j^{-1}g_ig_j\\
&=\big(b_i^{t_i}\big)^{-1}t_i^{-1}\big(b_j^{t_j}\big)^{-1}t_j^{-1}(b_it_i)(b_jt_j)\\
&=-b_i^{t_i}-b_j^{t_it_j}+b_i^{t_it_j}+b_j^{t_j}\\
&=-(1+X_i)^{-1}b_i-(1+X_i)^{-1}(1+X_j)^{-1}b_j+(1+X_i)^{-1}(1+X_j)^{-1}b_i+(1+X_j)^{-1}b_j\\
&=(1+X_i)^{-1}(1+X_j)^{-1}(X_ib_j-X_jb_i).
\end{align*}
Hence $K(b)$ is the ideal generated by the coordinates of $\Psi(b)$.
\end{proof}

\begin{lemma}\label{Lem:KernelPsi}
Under Notation~$\ref{Ntn:2}$, $|\mathrm{Ker}(\Psi)|=p^{p^m-1}$.
\end{lemma}

\begin{proof}
By Lemma~\ref{Lem:RelationMap}, $\Psi(b)=0$ if and only if $K(b)=0$.
Since the restriction of the natural projection $\rho\colon B\rtimes T\to T$ to $G$ is an epimorphism with kernel $K(b)$, it follows that $\Psi(b)=0$ if and only if $G(b)$ is a complement to $B$ in $B\rtimes T$.
Moreover, every complement of $B$ in $B\rtimes T$ contains a unique element in $\rho^{-1}(t_i)$ for each $i\in\{1,\ldots,m\}$, and these elements have the form $b_it_i$ for a unique $b=(b_1,\ldots,b_m)\in B^m$ and generate the complement.
Hence every complement is $G(b)$ for a unique $b$, and so the elements of $\mathrm{Ker}(\Psi)$ are in bijection with complements of $B$ in $B\rtimes T$.

Let $C$ be a complement of $B$ in $B\rtimes T$. If an element of $C$ fixes some point of $V\times\bbF_p$, then its image in $T$ fixes a point of $V$, and since $T$ acts regularly on $V$, this element belongs to $C\cap B=1$. Hence $C$ is semiregular on $V\times\bbF_p$, and so each orbit of $C$ has size $|C|=p^m$. Since the image of $C$ in $T$ is regular on $V$, it follows that each orbit of $C$ has the form
\[
\{(v,a_v)\mid v\in V\}.
\]
Let $f\in B$ be defined by $f(v)=-a_v$ for $v\in V$. Then $C^f$ stabilizes $V\times\{0\}$. An element of $B\rtimes T$ has the form $tc$ for some $t\in T$ and $c\in B$, and such an element stabilizes $V\times\{0\}$ if and only if $c=0$. Consequently, $C^f=T$. This proves that every complement of $B$ in $B\rtimes T$ is conjugate to $T$ by some element of $B$.

By the above two paragraphs, the elements of $\mathrm{Ker}(\Psi)$ are in bijection with the conjugates of $T$ by elements of $B$. Since $B$ is normalized by $T$ with \(B\cap T=1\), an element of $B$ normalizes $T$ if and only if it centralizes $T$. Moreover, it is straightforward to verify that $\Cen_B(T)$ is the group of constant functions from $V$ to $\bbF_p$. Therefore,
\[
|\mathrm{Ker}(\Psi)|=\frac{|B|}{|\Nor_B(T)|}=\frac{|B|}{|\Cen_B(T)|}=\frac{|B|}{p}=p^{p^m-1},
\]
as required.
\end{proof}

\begin{lemma}\label{Lem:ManyIdeals}
Under Notation~$\ref{Ntn:2}$, let $d_1$ and $d_2$ be integers such that $0\leq d_1\leq d_2\leq(p-1)m$ and $d_2-d_1+p-2\leq(p-1)m$, and let
\[
U=\bigoplus_{d=d_1}^{d_2}R_d
\ \text{ and }\
W=\bigoplus_{d=0}^{d_2-d_1+p-2}R_d.
\]
If $m\dim_{\bbF_p}(U)\geq p^m$, then at least
\[
p^{m\dim_{\bbF_p}(U)-p^m-\binom{m+1}{2}^2\dim_{\bbF_p}(W)}
\]
distinct nonzero ideals of $R$ occur among the ideals $K(b)$ with $b\in U^m$.
\end{lemma}

\begin{proof}
Write the coordinates of $\Psi(b)$ as $\Psi_1(b),\ldots,\Psi_r(b)$ for $b\in U^m$, where $r=\binom{m+1}{2}$.

Fix $b\in U^m$, and suppose that $c\in U^m$ satisfies $K(c)=K(b)$. By Lemma~\ref{Lem:RelationMap}, for each $i\in\{1,\ldots,r\}$, there exist $\lambda_{i,1},\ldots,\lambda_{i,r}\in R$ such that
\begin{equation}\label{Eqn:1}
\Psi_i(c)=\sum_{j=1}^r\lambda_{i,j}\Psi_j(b).
\end{equation}
By the definition of $\Psi$, each coordinate of $\Psi(b)$ and $\Psi(c)$ belongs to
\[
\bigoplus_{d=d_1+1}^{d_2+p-1}R_d.
\]
Accordingly, the part of the right hand side of~\eqref{Eqn:1} of degree at most $d_2+p-1$ depends only on the homogeneous components of the $\lambda_{i,j}$ of degree at most
\[
(d_2+p-1)-(d_1+1)=d_2-d_1+p-2.
\]
Thus, as $c$ varies subject to $K(c)=K(b)$, there are at most
\[
p^{r^2\dim_{\bbF_p}(W)}
\]
possibilities for $\Psi(c)$. For each fixed value of $\Psi(c)$, Lemma~\ref{Lem:KernelPsi} shows that there are at most $p^{p^m-1}$ possibilities for $c$. Consequently, for each fixed $b\in U^m$, there are at most
\begin{equation}\label{Eqn:5}
p^{p^m-1+r^2\dim_{\bbF_p}(W)}
\end{equation}
elements $c\in U^m$ such that $K(c)=K(b)$.

There are $p^{m\dim_{\bbF_p}(U)}$ elements $b$ of $U^m$, and Lemmas~\ref{Lem:RelationMap} and~\ref{Lem:KernelPsi} show that at most $p^{p^m-1}$ of them satisfy $K(b)=0$. If $m\dim_{\bbF_p}(U)\geq p^m$, then at least
\[
p^{m\dim_{\bbF_p}(U)}-p^{p^m-1}\geq p^{m\dim_{\bbF_p}(U)}-p^{m\dim_{\bbF_p}(U)-1}=(p-1)p^{m\dim_{\bbF_p}(U)-1}\geq p^{m\dim_{\bbF_p}(U)-1}
\]
elements $b$ of $U^m$ give a nonzero $K(b)$. Dividing this number by~\eqref{Eqn:5} then gives the conclusion of the lemma.
\end{proof}

\begin{theorem}\label{Thm:LowerBoundDetail}
For each fixed prime number $p$, as $n$ increases as a power of $p$, the number of permutational isomorphism classes of minimally transitive permutation groups of degree $n$ is at least $2^{(1/p-o(1))n\log(n)}$.
\end{theorem}

\begin{proof}
Adopt Notation~\ref{Ntn:2}, and let
\[
d_1=\left\lceil\frac{(p-1)m}{2}-m^{2/3}\right\rceil
\ \text{ and }\
d_2=\left\lfloor\frac{(p-1)m}{2}+m^{2/3}\right\rfloor.
\]
For sufficiently large $m$, the hypotheses $0\leq d_1\leq d_2\leq(p-1)m$ and $d_2-d_1+p-2\leq(p-1)m$ in Lemma~\ref{Lem:ManyIdeals} are satisfied. Let $h=d_2-d_1+p-2$, and let
\[
U=\bigoplus_{d=d_1}^{d_2}R_d
\ \text{ and }\
W=\bigoplus_{d=0}^{h}R_d.
\]
Note that $h\leq2m^{2/3}+p=O(m^{2/3})$.

We first estimate the dimensions of $U$ and $W$. Let $Y_1,\ldots,Y_m$ be independent random variables, each uniformly distributed on $\{0,1,\ldots,p-1\}$. The monomial basis of $R$ gives
\[
\frac{\dim_{\bbF_p}(U)}{p^m}
=\Pr(d_1\leq Y_1+\cdots+Y_m\leq d_2).
\]
Since $Y_1+\cdots+Y_m$ has mean $(p-1)m/2$ and variance $m(p^2-1)/12$,
\[
1-\frac{\dim_{\bbF_p}(U)}{p^m}
=\Pr\left(\left|Y_1+\cdots+Y_m-\frac{(p-1)m}{2}\right|>m^{2/3}\right)
\leq\frac{m(p^2-1)}{12m^{4/3}}
\]
by Lemma~\ref{Lem:Chebyshev}. Therefore,
\begin{equation}\label{Eqn:2}
\dim_{\bbF_p}(U)=(1-o(1))p^m.
\end{equation}
Since the number of monomials in $X_1,\ldots,X_m$ of degree at most $h$ is at most $\binom{m+h}{h}$,
\[
\dim_{\bbF_p}(W)\leq\binom{m+h}{h}\leq(m+h)^h.
\]
Then we derive from $h=O(m^{2/3})$ that
\[
\log\big(\dim_{\bbF_p}(W)\big)=O(m^{2/3}\log(m))=o(m),
\]
and hence
\begin{equation}\label{Eqn:3}
\binom{m+1}{2}^2\dim_{\bbF_p}(W)=p^{o(m)}=o(p^m).
\end{equation}

Now~\eqref{Eqn:2} implies that $m\dim_{\bbF_p}(U)\geq p^m$ for sufficiently large $m$, and it follows from Lemma~\ref{Lem:ManyIdeals} together with~\eqref{Eqn:2} and~\eqref{Eqn:3} that at least
\[
p^{m\dim_{\bbF_p}(U)-p^m-\binom{m+1}{2}^2\dim_{\bbF_p}(W)}=p^{(1-o(1))mp^m}
\]
distinct nonzero ideals of $R$ occur as $K(b)$ for some $b\in U^m$. For each such choice of $b$, Lemma~\ref{Lem:K} asserts that $G(b)$ is minimally transitive on $V\times\bbF_p$, and that
\[
\Phi(G(b))=K(b)
\]
has $\mathcal P$ as its set of orbits on $V\times\bbF_p$.

Let $b,c\in U^m$ be such that $K(b)$ and $K(c)$ are nonzero, and that $G(b)$ is conjugate to $G(c)$ by some permutation $\alpha$ of $V\times\bbF_p$. Then
\begin{equation}\label{Eqn:6}
K(b)^\alpha=\Phi(G(b))^\alpha=\Phi(G(b)^\alpha)=\Phi(G(c))=K(c).
\end{equation}
In particular, $\alpha$ preserves $\mathcal P$. Let $\overline{\alpha}$ be the permutation of $\mathcal P$ induced by $\alpha$. Since
\[
G(b)^\alpha/K(b)^\alpha=G(c)/K(c),
\]
the permutation $\overline{\alpha}$ normalizes $T$. Identifying $\mathcal P$ with $V$ via $\{v\}\times\bbF_p\mapsto v$, we deduce that
\begin{equation}\label{Eqn:10}
\overline{\alpha}\in\Nor_{\Sym(V)}(T)=T\rtimes\GL(V).
\end{equation}
Since for each $v\in V$, the permutation groups induced by $K(b)$ and by $K(c)$ on $\{v\}\times\bbF_p$ are both $\bbF_p^+$, it follows from $\Nor_{\Sym(\bbF_p)}(\bbF_p^+)=\bbF_p^+\rtimes\GL_1(p)$ together with~\eqref{Eqn:6} and~\eqref{Eqn:10} that
\begin{equation}\label{Eqn:11}
\alpha\in\big(B\rtimes\GL_1(p)^{p^m}\big)\rtimes(T\rtimes\GL(V)).
\end{equation}
Note that $K(b)$ is centralized by $B$ as $K(b)\leq B\cong(\bbF_p^+)^{|V|}$. Moreover, since $K(b)$ is an $\bbF_p[T]$-submodule of $B$, every element of $T$ normalizes $K(b)$. Then~\eqref{Eqn:11} implies that, for each fixed $b$ such that $K(b)$ is nonzero, there are at most
\[
|\GL_1(p)^{p^m}||\GL(V)|=(p-1)^{p^m}|\GL(V)|<p^{p^m}|\GL_m(p)|<p^{p^m+m^2}
\]
distinct subgroups $K(c)$ such that $G(c)$ is permutationally isomorphic to $G(b)$.

Combining the consequences of the above two paragraphs, we conclude that the number of minimally transitive permutation groups of degree $n=p^{m+1}$, counted up to permutational isomorphism, is at least
\[
p^{(1-o(1))mp^m-(p^m+m^2)}=p^{(1-o(1))mp^m}=2^{(1/p-o(1))n\log(n)}.
\]
This completes the proof.
\end{proof}

\section{Concluding remarks}\label{Sec:Remark}

In view of Theorems~\ref{Thm:UpperBound} and~\ref{Thm:LowerBound}, the following two questions arise naturally.

\begin{question}
What is the infimum of positive real numbers $c$ such that the number of minimally transitive subgroups of $S_n$ is at most $2^{(c+o(1))n\log(n)}$?
\end{question}

\begin{question}
What is the infimum of positive real numbers $c$ such that the number of permutational isomorphism classes of minimally transitive permutation groups of degree $n$ is at most $2^{(c+o(1))n\log(n)}$?
\end{question}

\begin{remark}
From Theorems~\ref{Thm:LowerBoundDetail} and~\ref{Thm:UpperBoundDetail} we know that the infima in the above two questions are both between $1/2$ and $3+\log(3)/3\approx3.52832$.
\end{remark}

As mentioned in the Introduction, the results established in Section~\ref{Sec:LowerBound} enable us to construct minimally transitive permutation groups of large order, as recorded in Theorem~\ref{Thm:LargeOrder}.

\begin{proof}[\bf\textup{Proof of Theorem~\ref{Thm:LargeOrder}}]
The theorem holds readily for $n=p^2$ since any regular permutation group of degree $p^2$ is minimally transitive.
Now adopt Notation~\ref{Ntn:2}, let $n=p^{m+1}\geq p^3$, and take $b=(1,0,\ldots,0)\in R^m$.
By Lemma~\ref{Lem:RelationMap}, the coordinates of $\Psi(b)$ generate the ideal
\[
K(b)=(X_1^{p-1},X_2,\ldots,X_m).
\]
It follows that $R/K(b)\cong\bbF_p[X_1]/(X_1^{p-1})$, and so
\[
\dim_{\bbF_p}(K(b))=p^m-(p-1).
\]
In particular, $K(b)$ is nontrivial, and then Lemma~\ref{Lem:K} asserts that $G(b)$ is a minimally transitive permutation group on $V\times\bbF_p$ of degree $p^{m+1}=n$. Moreover, since $G(b)/K(b)\cong T$,
\[
|G(b)|=|K(b)||T|=p^{p^m-(p-1)}\cdot p^m=np^{n/p-p}.
\]
This completes the proof.
\end{proof}

The asymptotic enumeration of vertex-transitive graphs and digraphs naturally gives rise to four problems: one may consider graphs or digraphs, and in either case one may count labelled or unlabelled (up to isomorphism) objects.
For unlabelled graphs, the McKay--Praeger conjecture asserts that asymptotically almost all vertex-transitive graphs are Cayley graphs~\cite[p.~54]{MP1994}. Analogous conjectures may be formulated for each of the other three versions.
Before considering these limiting proportions, a basic problem is to determine the order of magnitude of the number of vertex-transitive objects in each of the four versions. The upper bound for minimally transitive groups in this paper, together with an elementary lower bound, determines this order of magnitude in the labelled versions, for both graphs and digraphs, as Corollary~\ref{Cor:VT} states.

\begin{proof}[\bf\textup{Proof of Corollary~\ref{Cor:VT}}]
Let $\mathcal{M}_n$ be the set of minimally transitive subgroups of $S_n$. Every vertex-transitive graph or digraph on $\{1,\ldots,n\}$ admits some member of $\mathcal{M}_n$ as a group of automorphisms.
For each $G\in\mathcal{M}_n$, the orbitals of $G$ are in bijection with the orbits of a point stabilizer in $G$ on $\{1,\ldots,n\}$. Hence $G$ has at most $n$ orbitals, and so at most $2^n$ digraphs on $\{1,\ldots,n\}$ admit $G$ as a group of automorphisms. It follows from Theorem~\ref{Thm:UpperBoundDetail} that the number of labelled vertex-transitive digraphs of order $n$ is at most
\[
|\mathcal{M}_n|2^n\leq2^{(3+\log(3)/3+o(1))n\log(n)}\cdot2^n=2^{O(n\log(n))}.
\]
For the lower bound, suppose that $n\geq3$. The cycle of length $n$ has automorphism group of order $2n$, and hence has
\[
\frac{n!}{2n}=2^{(1-o(1))n\log(n)}
\]
distinct labellings. This gives the required lower bound and completes the proof.
\end{proof}

\begin{remark}
The proof of Corollary~\ref{Cor:VT} actually gives the upper bound $2^{(3+\log(3)/3+o(1))n\log(n)}$ for the number of labelled vertex-transitive digraphs, and hence also for the number of labelled vertex-transitive graphs. On the other hand, the labelled cycles used in the proof provide only the lower bound $2^{(1-o(1))n\log(n)}$ for the Cayley subfamilies. Thus, although these estimates determine the order of magnitude in Corollary~\ref{Cor:VT}, they do not show that asymptotically almost all labelled vertex-transitive graphs or digraphs are Cayley.
\end{remark}

For the unlabelled digraph problem, Morris and Spiga~\cite[Subsection~8.4]{MS2021} proposed two possible group-theoretic routes: estimating, up to permutational isomorphism, either the minimally transitive groups or the transitive $2$-closed groups of degree $n$. They observed that an upper bound of $2^{o(n)}$ for either class, together with the fact that a nonregular transitive group of degree $n$ has at most $3n/4$ orbitals, would give at most $2^{3n/4+o(n)}$ non-Cayley vertex-transitive digraphs up to isomorphism. Since there are at least $2^{n+o(n)}$ Cayley digraphs up to isomorphism, this would prove the digraph analogue of the McKay--Praeger conjecture.

Clearly, the upper bound $2^{o(n)}$ proposed by Morris and Spiga could be relaxed to $2^{cn+o(n)}$ for any constant $c<1/4$, while still sufficing to prove the digraph analogue of the McKay--Praeger conjecture. However, Theorem~\ref{Thm:LowerBound} shows that even such a relaxed estimate cannot hold for minimally transitive groups. Note that this does not affect the route through transitive $2$-closed groups. Also, the minimally transitive route might instead be refined by weighting each permutational isomorphism class according to the number of non-Cayley digraphs arising as unions of its orbitals. Our result shows that the unweighted cardinality is too large for this argument, but leaves such a weighted estimate open.

\section*{Acknowledgments}

We are deeply indebted to Professor L\'{a}szl\'{o} Pyber for communicating his proof of Theorem~\ref{Thm:UpperBound} and for suggesting that we consider minimally transitive permutation groups of large order. At an early stage of this project, he also suggested that we search for examples within wreath products. This suggestion motivated the search that led to the construction underlying our lower bound, with extensive computations in \textsc{Magma}~\cite{BCP1997} playing an important role in its subsequent development. The second author was supported by the National Research, Development and Innovation Office (NKFIH) Grant No.~K153681.

\end{document}